\documentclass{amsart}
\usepackage{amsaddr}
\usepackage{hyperref}
\usepackage{url}

\newtheorem{theorem}{Theorem}

\newtheorem{lemma}{Lemma}

\newcommand{\Exk}{\exists\, k\ge 1: \:}
\newcommand{\mor}[2]{\Big\{
   \begin{aligned} 0 & \rightarrow #1\\[-.1cm] 1 & \rightarrow #2  \end{aligned}}

\title[A remarkable sequence]{A remarkable integer sequence related to $\pi$ and $\sqrt{2}$}
\author{Wieb Bosma, Michel Dekking, Wolfgang Steiner}
\address{\rm
Radboud University Nijmegen,\\
P.O.\ Box 9010, \\
6500 GL Nijmegen, the Netherlands;\\
{\tt bosma@math.ru.nl}\\
\sl\and\\
\rm
Delft University of Technology, \\Faculty EEMCS, P.O.~Box 5031,\\ 2600 GA
Delft, the Netherlands;\\
{\tt F.M.Dekking@math.tudelft.nl}\\
\sl\and\\
\rm
IRIF, CNRS UMR 8243,\\
Universit\'e Paris Diderot -- Paris 7,\\
Case 7014, 75205 Paris Cedex~13, France;\\
{\tt steiner@irif.fr}
}

\date{\today}
\begin{document}

\begin{abstract}
We prove that five ways to define entry {\bf A086377} in the OEIS
do lead to the same integer sequence.


\medskip


\end{abstract}

\maketitle

\section{Introduction}
\noindent
In September of 2003 Benoit Cloitre contributed a
sequence to the On-Line Encyclopedia of Integer Sequences \cite{A086377},
defined by him as $a_1=1$, and for $n\geq 2$ by
\begin{equation}
a_n= \left\{ \begin{aligned}
a_{n-1}+2&\quad{\rm{if\:}} n \: {\rm is\: in \: the \: sequence,}\\
a_{n-1}+2&\quad{\rm{if\:}} n \: {\rm and\:} n\!-\!1 \: {\rm are\: not \: in \: the \: sequence,}\\
a_{n-1}+3&\quad{\rm{if\:}} n \: {\rm is\: not\: in \: the \: sequence, but\:} n\!-\!1 \: {\rm is\: in \: the \: sequence.}
\end{aligned}\right.
\end{equation}
The first 25 values of this sequence are
$$1, 4, 6, 8, 11, 13, 16, 18, 21, 23, 25, 28, 30, 33, 35, 37, 40, 42, 45, 47, 49, 52, 54, 57, 59.$$
The purpose of this paper is to prove equivalence of five
ways to define this integer sequence, most of them already conjecturally
stated in the OEIS article on {\bf A086377}. Besides a simplified recursion, the alternatives
involve statements in terms of a morphic sequence, of a Beatty sequence, and of
approximation properties linking a classical continued fraction of $\frac{4}{\pi}$ to
that of $\sqrt{2}$.
\section{The theorem}
\begin{theorem}\label{thm}
The following five definitions produce the same integer sequence:
\begin{itemize}
\item[$(a_n)$] defined by $a_1=1$ and for $n\geq 2$:
\begin{equation*}\label{eq:A086377}
a_n= \left\{ \begin{aligned}
a_{n-1}+2&\quad{\rm{if\:}} n \: {\rm is\: in \: the \: sequence,}\\
a_{n-1}+2&\quad{\rm{if\:}} n \: {\rm and\:} n\!-\!1 \: {\rm are\: not \: in \: the \: sequence,}\\
a_{n-1}+3&\quad{\rm{if\:}} n \: {\rm is\: not\: in \: the \: sequence, but\:} n\!-\!1 \: {\rm is\: in \: the \: sequence.}
\end{aligned}\right.
\end{equation*}
\item[$(b_n)$] defined by $b_1=1$ and for $n\geq 2$:
\begin{equation*}\label{eq:short}
b_n= \left\{ \begin{aligned}
b_{n-1}+2&\quad{\rm{if\:}} n\!-\!1\: {\rm is\: not \: in \: the \: sequence,}\\
b_{n-1}+3&\quad{\rm{if\:}} n\!-\!1 \: {\rm is\: in \: the \: sequence.}
\end{aligned}\right.
\end{equation*}
\item[$(c_n)$] for $n\geq 1$ defined as the position of the $n$-th zero in
the fixed point of the morphism
\begin{equation*}\label{eq:morphism}
\phi:\ \left\{ \begin{aligned}
0&\mapsto 011\\
1&\mapsto 01
\end{aligned}\right.;
\end{equation*}
\item[$(d_n)$] defined by $d_n=\left\lfloor (1+\sqrt{2})\cdot n -\frac12\sqrt{2}\right\rfloor$ for $n\geq 1$;
\item[$(e_n)$] defined by $e_n=\lceil r_n\rfloor=\lfloor r_n+\frac{1}{2}\rfloor$, with $r_1=\dfrac{4}{\pi}$ and $r_{n+1}=\dfrac{n^2}{r_{n}-(2n-1)}$, for $n\geq 1$.
\end{itemize}
\end{theorem}
\noindent
At first we found it hard to believe the equivalence of these definitions, but
a verification of the first 130000 terms ($a_{130000}=
313847$) convinced us to look for proofs.

\section{Simplification and a morphic sequence}
\noindent
To show that $(b_n)$ defines the same sequence as $(a_n)$, simply note
that $a_n-a_{n-1}\ge 2$ for all $n$: hence if $n$ is in the sequence
then $n-1$ is not, and we can combine the first two cases
in Equation~(\ref{eq:A086377}).

In a comment to sequence {\bf A086377}, Clark Kimberling asked if the integers
in this sequence coincide with the positions of the zeroes in sequence
{\bf A189687}, which is the fixed point of the substitution
\begin{equation*}
\phi:\ \left\{ \begin{aligned}
0&\mapsto 011\\
1&\mapsto 01
\end{aligned}\right.,
\end{equation*}
defining the sequence $(c_n)$ in the Theorem.
It is not hard to see that this indeed produces the same as sequence
$(b_n)$; repeatedly applying the morphism $\phi$ to $0$ produces after a few steps
the initial segment
$$0110101011010110101101010110101101010110101101010110101101011\cdots.$$
The position $c_n$ of the $n$-th zero is 2 ahead of $c_{n-1}$ precisely
when the latter is followed by a single~1, that is, when there is a 1
at position~$n-1$, and it is 3 ahead of $c_{n-1}$ if that zero is followed by~11, which means that there was a 0 at position~$n-1$. Thus the rule is
exactly that defining~$(b_n)$.

\section{Beatty sequence}
\noindent
Every pair of real numbers $\alpha$ and $\beta$ determines a Beatty sequence by
$${\rm B}(\alpha, \beta)_n :=\lfloor n\alpha+\beta \rfloor, \qquad n=1,2,\dots .$$
The numbers $\alpha$ and $\beta$ also determine sequences by
$$  {\rm St}(\alpha, \beta)_n :=\lfloor (n+1)\alpha+\beta \rfloor-\lfloor n\alpha+\beta \rfloor, \qquad n=1,2,\dots ,$$
which is a Sturmian sequence (of slope $\alpha$), over the alphabet $\{0, 1\}$,
provided that $0\leq\alpha<1$.

Thus Sturmian sequences are first differences of Beatty sequences (when $0\le \alpha<1$), but Beatty sequences and Sturmian sequences are also linked in another way.

\begin{lemma}\label{lem:pos1} Let $\alpha>1$ be irrational, and let $(s_n)_{n\ge 1}$ be given by
$s_n={\rm St}(\frac{1}{\alpha}, -\frac{\beta}{\alpha})_n$, for some real number $\beta$ with $\alpha+\beta> 1$ and such that $k\alpha+\beta \not\in \mathbb{Z}$
for all positive integers $k$. Then ${\rm B}(\alpha, \beta)$
is the sequence of positions of $1$ in $(s_n)$.
\end{lemma}

\begin{proof} This is a generalization of Lemma 9.1.3 in \cite{Allou}, from homogeneous to inhomogeneous Sturmian sequences.
 The proof also generalizes:
\begin{eqnarray*}
\Exk n = \lfloor k\alpha + \beta\rfloor & \Longleftrightarrow & \Exk n \le  k\alpha + \beta < n+1\\
& \Longleftrightarrow & \Exk  \frac{n-\beta}{\alpha} \le k < \frac{n+1-\beta}{\alpha}\\
& \Longleftrightarrow & \Exk \bigg\lfloor\frac{n-\beta}{\alpha}\bigg\rfloor=k-1 {\rm \: and \:} \bigg\lfloor\frac{n-\beta}{\alpha}+\frac1{\alpha}\bigg\rfloor=k\\
& \Longleftrightarrow & \bigg\lfloor\frac{n+1}{\alpha}-\frac{\beta}{\alpha}\bigg\rfloor- \bigg\lfloor\frac{n}{\alpha}-\frac{\beta}{\alpha}\bigg\rfloor=1\\
& \Longleftrightarrow & {\rm St}\bigg(\frac1\alpha, -\frac{\beta}{\alpha}\bigg)_n=1. \hspace{12em} \qedhere
\end{eqnarray*}
\end{proof}
\noindent
Our goal in this section is to prove that $(c_n)=(d_n)$. 
Let $\psi$ be the morphism $\psi: \mor{10}{100}$, and
let $w$ be the fixed point. Then
$$w = 1001010100101001010010101001010010101001010010101001010010100\cdots,$$
which is the mirror image  of $\phi$ in the definition of $(c_n)$, i.e., $\psi=E\phi E$, with $E$ the exchange morphism given by $E(0)=1, E(1)=0$.
So the positions of $0$ in the fixed point of $\phi$ correspond to the positions of $1$ in the fixed point $w$ of $\psi$.

Let $\alpha_d=1+\sqrt{2}$ and $\beta_d=-\frac12\sqrt{2}$; then
$d_n={\rm B}(\alpha_d, \beta_d)_n$, for $n\geq1$.

Applying  Lemma~\ref{lem:pos1}, we deduce that $d_n$ also equals the position
of the $n$-th $1$ in the Sturmian sequence ${\rm St}(\alpha,\beta)$, generated by
$$\alpha=\frac1{\alpha_d}=\sqrt{2}-1, \; \beta=\frac{-\beta_d}{\alpha_d}= 1-\frac12\sqrt{2}.   $$

\begin{lemma}\label{lem:psi} ${\rm St}\big(\sqrt{2}-\!1, 1-\frac12\sqrt{2}\big)=w$. 
\end{lemma}

\begin{proof} This was already proved by Nico de Bruijn in 1981 (\cite{Bruijn de}), where it is the main example. Note, however, that our Sturmian sequences
start at $n=1$.

For a `modern' proof as suggested by \cite[Section 4]{Dekk-Stur}, let $\psi_1$ and $\psi_2$ be the elementary morphisms given by $\psi_1(0)=01, \psi_1(1)=0$, and $\psi_2(0)=10, \psi_2(1)=0$.
Then $\psi=\psi_2\psi_1E$. This implies that  the fixed point $w$ of $\psi$ is a Sturmian word (see \cite[Corollary 2.2.19]{Lothaire}). To find its parameters $(\alpha,\beta)$, use the 2D fractional linear maps that describe how the parameters of a Sturmian word change when one applies an elementary morphism.
For Sturmian words starting at $n=0$, the maps for $E, \psi_1$ and $\psi_2$ are\footnote{Actually there is a subtlety here involving the ceiling representation of a Sturmian sequence, but that does not apply in our case since $\beta \not\in \mathbb{Z} \alpha + \mathbb{Z}$.} respectively (see \cite[Lemma 2.2.17, Lemma 2.2.18, Exercise 2.2.6]{Lothaire})
\begin{equation*}T_0(x,y)=(1-x,1-y),\; T_1(x,y)=\left(\frac{1-x}{2-x},\frac{1-y}{2-x}\right), \;
T_2(x,y)= \left(\frac{1-x}{2-x},\frac{2-x-y}{2-x}\right).
\end{equation*}
The change of parameters by applying $\psi$ is therefore the composition
\begin{equation*}  T_{210}(x,y):=T_2T_1T_0(x,y)= \left(\frac{1}{2+x},\frac{2+x-y}{2+x}\right).
\end{equation*}
But the parameters $\alpha$ and $\beta$ of $w$ do \emph{not} change when one applies $\psi$. This means that $(\alpha,\beta)$ is a fixed point of $T_{210}$, and one easily computes $\alpha=\sqrt{2}-1$, and then $\beta=\frac12 \sqrt{2}$.
Since our Sturmian words start at $n=1$, we have to subtract $\alpha$ from $\beta$ and obtain that $w = {\rm St}\big(\sqrt{2}-\!1, 1-\frac12\sqrt{2}\big)$.
\end{proof}

\section{Converging recurrence}
\noindent
In a comment to entry {\bf A086377}, Joseph Biberstine conjectured a
beautiful connection with the infinite continued fraction expansion
$$\frac{4}{\pi}=
1+{\strut 1^2\over
\displaystyle{3+{\strut 2^2\over
\displaystyle{5+{\strut 3^2\over
\displaystyle{7+{\strut 4^2\over
\displaystyle{9+{\strut 5^2\over
\displaystyle{11+{\strut 6^2\over\ddots}}}}}}}}}}},$$
derived from the arctangent function expansion.
If we define $R_n$ for $n\geq 1$ by
$$R_n=
2n-1+{\strut n^2\over
\displaystyle{2n+1+{\strut (n+1)^2\over
\displaystyle{2n+3+{\strut (n+2)^2\over
\displaystyle{2n+5+{\strut (n+3)^2\over\ddots}}}}}}},$$
then $R_1 = 4/\pi$ and $\displaystyle R_n=2n-1+\frac{n^2}{R_{n+1}}$.
We see that
$$\frac{R_n}{n}\frac{R_{n+1}}{n+1}-\frac{2n-1}{n}\frac{R_{n+1}}{n+1}-
\frac{n^2}{n(n+1)}=0.$$
This implies that if $R_n/n$ converges, for $n\rightarrow\infty$,
then it does so to a (positive) zero of $x^2-2x-1$, that is, to $1+\sqrt{2}$;
cf.~Lemma \ref{lemma:rn} below.

We consider now, conversely and slightly more generally, for any real $h\geq 1$,
a sequence of positive numbers $r_n$ satisfying
\begin{equation} \label{e:r}
r_n = h n - 1 + \frac{n^2}{r_{n+1}}
\end{equation}
for $n \ge 1$.
We first show that this sequence is unique, i.e., there is a unique $r_1 > 0$ such that $r_n > 0$ for all $n \ge 1$, and give estimates for its terms.

\begin{lemma}\label{lemma:rn}
For each $h \ge 1$, there is a unique sequence of positive real numbers $(r_n)_{n\ge1}$ satisfying the recurrence~\eqref{e:r}.
Moreover, we have for this sequence, for all $n \ge 1$,
\begin{equation} \label{e:rn}
0 < r_n - \alpha n + c < \frac{(\alpha-c)(c-1)}{\alpha n}
\end{equation}
with $\alpha = \dfrac{h+\sqrt{h^2+4}}{2}$ and $c = \dfrac{1+\alpha}{2\alpha-h} = \dfrac{1}{2} + \dfrac{h+2}{2\sqrt{h^2+4}}$.
\end{lemma}

\begin{proof}
Let $f_n(x) = hn-1+n^2/x$.
Suppose that a sequence of positive numbers $r_n$ satisfies~\eqref{e:r}, i.e., that $f_n(r_{n+1}) = r_n$ for all $n \ge 1$.
Then we have $r_n > hn-1$ and thus $r_n < (h+1/h) n$ for all $n \ge 1$.
We deduce that there exists some $\delta > 0$ and $N \ge 1$ such that $r_n > (h+\delta) n$ for all $n \ge N$.
Suppose that there is another sequence of positive numbers $\tilde{r}_n$ satisfying~\eqref{e:r}.
Since $|f_n'(x)| = |n/x|^2 < 1/(h+\delta)$ for all $x > (h+\delta)n$, we have
\[
|r_N - \tilde{r}_N| = |f_N f_{N+1} \cdots f_{n-1}(r_n) - f_N f_{N+1} \cdots f_{n-1}(\tilde{r}_n)| < \frac{|r_n - \tilde{r}_n|}{(h+\delta)^{n-N}} < \frac{n/h}{(h+\delta)^{n-N}}
\]
for all $n \ge N$, hence $r_N = \tilde{r}_N$, which implies that $r_n = \tilde{r}_n$ for all $n \ge 1$.

Next we show that
\[
f_n\big(\alpha (n+1) - c\big) < \alpha n - c + \frac{(\alpha-c)(c-1)}{\alpha n}
\]
and
\[
f_n\Big(\alpha (n+1) - c + \frac{(\alpha-c)(c-1)}{(n+1)\alpha}\Big) > \alpha n - c.
\]
Indeed, using that $\alpha^2 = h \alpha + 1$ and $2\alpha c - h c = 1 + \alpha$, we have
\begin{align*}
(\alpha n + \alpha - c)\, f_n\big(\alpha (n+1) - c\big) & = (hn-1) (\alpha n + \alpha - c)  + n^2 \\
& = (h\alpha+1) n^2 + (h\alpha-hc-\alpha) n - (\alpha-c) \\
& < \alpha^2 n^2 + (\alpha^2 - 2\alpha c) n - (\alpha-c) + \frac{(\alpha-c)^2(c-1)}{\alpha n} \\
& = (\alpha n + \alpha - c) \Big(\alpha n - c + \frac{(\alpha-c)(c-1)}{\alpha n}\Big),
\end{align*}
and
\begin{align*}
\Big(\alpha n + \alpha - c + \frac{(\alpha-c)(c-1)}{\alpha(n+1)}\Big) (\alpha n - c) \hspace{-9em} & \hspace{9em} < \alpha^2 n^2 + (\alpha^2 - 2\alpha c) n - (\alpha-c) - \frac{c(\alpha-c)(c-1)}{\alpha(n+1)} \\
& <  (h\alpha+1) n^2 + (h\alpha-hc-\alpha) n - (\alpha-c) - \frac{(\alpha-c)(c-1)}{\alpha(n+1)} \\
& < (hn-1) \Big(\alpha n + \alpha - c + \frac{(\alpha-c)(c-1)}{\alpha(n+1)}\Big)  + n^2 \\
& = \Big(\alpha n + \alpha - c + \frac{(\alpha-c)(c-1)}{\alpha(n+1)}\Big)\, f_n\Big(\alpha (n+1) - c + \frac{(\alpha-c)(c-1)}{\alpha(n+1)}\Big).
\end{align*}
As $f_n$ is monotonically decreasing for $x > 0$, we deduce that
\[
0 < f_n(x) - \alpha n + c < \frac{(\alpha-c)(c-1)}{\alpha n}
\]
for all $x$ with $0 \le x - \alpha (n+1) + c \le \frac{(\alpha-c)(c-1)}{\alpha(n+1)}$.
Then we also have
\[
0 < f_n f_{n+1} \cdots f_{n+k-1}\big(\alpha (n+k) - c + x\big) - \alpha n + c < \frac{(\alpha-c)(c-1)}{\alpha n}
\]
for all $k,n \ge 1$, $0 \le x - \alpha (n+k) + c \le \frac{(\alpha-c)(c-1)}{\alpha(n+k)}$.
As $f_n$ is contracting for $x \ge \alpha(n+1)-c$, the intervals $[f_1 f_2 \cdots f_n(\alpha (n+1) -c),  f_1 f_2 \cdots f_n(\alpha (n+1) -c+\frac{(\alpha-c)(c-1)}{\alpha(n+1)})]$ converge to a point~$r_1$.
Then the numbers $r_n$ given by \eqref{e:r} satisfy \eqref{e:rn} for all $n \ge 1$.
By the first paragraph of the proof, this is the unique sequence of positive numbers satisfying~\eqref{e:r}.
\end{proof}

\noindent
Now consider when $\alpha n - c + \frac{1}{2}$ is close to $\lceil \alpha n - c + \frac{1}{2}\rceil$.
Let $p_k/q_k$ be the convergents of the regular continued fraction $\alpha = [h;h,h,\ldots]$, i.e., $q_{-1} = 0$, $q_0 = 1$, $q_{k+1} = h q_k + q_{k-1}$ for $k \ge 1$, $p_k = q_{k+1}$.
Then we have
\[
q_k = \frac{\alpha^{k+1}+(-1)^k/\alpha^{k+1}}{\alpha+1/\alpha}
\]
and thus
\begin{equation} \label{e:qp}
q_k \alpha - p_k = \frac{(-1)^k}{\alpha^{k+1}}.
\end{equation}

\begin{lemma}
Let $h$ be a positive integer and $\alpha = \dfrac{h+\sqrt{h^2+4}}{2}$.
Then we have
\[
\lceil \alpha n \rceil - \alpha n =
\left\{\begin{array}{ll}j/\alpha^{2k} & \mbox{if}\ n = j q_{2k-1},\, k \ge 1,\, 1 \le j < \alpha^{2k}, \\[1ex]
(\alpha-1)/\alpha^{2k+1} & \mbox{if}\ n = q_{2k-1} + q_{2k},\, k \ge 0, \\[1ex]
(\alpha+1)/\alpha^{2k+2} & \mbox{if}\ n = q_{2k+1} -  q_{2k},\, k \ge 0,
\end{array}\right.
\]
and $n (\lceil \alpha n \rceil - \alpha n) \ge 1$ for all other $n \ge 1$.
\end{lemma}

\begin{proof}
The formulas for $n = j q_{2k-1}$, $n = q_{2k-1} + q_{2k}$ and $n = q_{2k+1} -  q_{2k}$ are immediate from~\eqref{e:qp}.
By \cite[Ch.~2, \S 5, Theorem~2]{RockettSzusz}, we have $n (\lceil \alpha n \rceil - \alpha n) \ge 1$ for all $n \ge 1$ that are not of the form $j q_k$, $1 \le j < \alpha / \sqrt{h}$, $q_k+q_{k-1}$ or $q_k-q_{k-1}$.
Since $\alpha q_{2k} - \lfloor \alpha q_{2k} \rfloor = 1 /\alpha^{2k+1}$, $\alpha (q_{2k} + q_{2k+1}) - \lfloor \alpha (q_{2k} + q_{2k+1}) \rfloor = (\alpha-1) /\alpha^{2k+2}$ and $\alpha (q_{2k} - q_{2k-1}) - \lfloor \alpha (q_{2k} - q_{2k-1}) \rfloor = (\alpha+1) /\alpha^{2k+1}$, we have $\lceil \alpha n \rceil - \alpha n > 1/2$ for $n = jq_{2k}$, $n = q_{2k} + q_{2k+1}$ and $n = q_{2k} - q_{2k-1}$.
If moreover $n \ge 2$, then we have thus $n (\lceil \alpha n \rceil - \alpha n) \ge 1$ for these $n$ as well.
Since $q_0 + q_{-1} = 1$, the case $n = 1$ has already been treated.
\end{proof}

\noindent
We obtain that
\[
n \big(\lceil \alpha n \rceil - \alpha n\big) =
\left\{\begin{array}{ll}\displaystyle\frac{j^2(1-1/\alpha^{4k})}{\sqrt{h^2+4}} & \mbox{if}\ n = j q_{2k-1},\, k \ge 1,\, 1 \le j < \alpha^{2k}, \\[1ex]
\displaystyle\frac{h-(\alpha-1)^2/\alpha^{4k+2}}{\sqrt{h^2+4}} & \mbox{if}\ n = q_{2k-1} + q_{2k},\, k \ge 0, \\
\displaystyle\frac{h-(\alpha+1)^2/\alpha^{4k+4}}{\sqrt{h^2+4}} & \mbox{if}\ n = q_{2k+1} -  q_{2k},\, k \ge 0.\end{array}\right.
\]
The worst case for $n = q_{2k-1} + q_{2k}$ or $n = q_{2k+1} -  q_{2k}$ is given by $n = q_{-1} + q_0 = 1$, hence
\[
n \big(\lceil \alpha n \rceil - \alpha n\big) \ge h+1-\alpha = 1-\frac{1}{\alpha}
\]
for all $n \ge 1$ such that $n \ne q_{2k-1}$ for all $k \ge 1$.

Now we come back to the case $h = 2$ and consider the distance of $\alpha n - c + \frac{1}{2}$ to the nearest integer above $\alpha n - c + \frac{1}{2}$.
Note that $c-\frac{1}{2} = \frac{1}{\sqrt{2}}$.
We have
\[
2 \Big(\Big\lceil \alpha n - \frac{1}{\sqrt{2}}\Big\rceil - \alpha n + \frac{1}{\sqrt{2}}\Big) = 2 \Big\lceil \alpha n - \frac{1}{\sqrt{2}}\Big\rceil -1 - \alpha (2n-1) \ge \lceil \alpha (2n-1)\rceil - \alpha (2n-1) > \frac{\alpha-1}{2\alpha n},
\]
where we have used that $q_{2k-1}$ is even for all $k \ge 1$, thus
\[
\Big\lceil \alpha n - \frac{1}{\sqrt{2}}\Big\rceil - \alpha n + \frac{1}{\sqrt{2}} > \frac{\alpha-1}{4\alpha n}.
\]
Since $(\alpha-c)(c-1) = \displaystyle\frac{1}{4\alpha}$, we have
\[
\alpha n - \frac{1}{\sqrt{2}} < r_n+\frac{1}{2} < \alpha n - \frac{1}{\sqrt{2}} +\frac{1}{4\alpha n} < \alpha n - \frac{1}{\sqrt{2}} +\frac{\alpha-1}{4\alpha n} < \Big\lceil \alpha n - \frac{1}{\sqrt{2}}\Big\rceil
\]
for all $n \ge 1$, thus $d_n = e_n$.
This completes the proof of Theorem \ref{thm}.

\medskip\noindent
We remark that $h = 2$ cannot be replaced by an arbitrary positive integer in the previous paragraph.
For example, for $h = 1$, we have $\alpha = \frac{1+\sqrt{5}}{2}$, $c = \frac{\alpha^2}{\sqrt{5}}$, $\lfloor 137 \alpha - c + \frac{1}{2} \rfloor = 220$ and $\lfloor r_{137} + \frac{1}{2} \rfloor = 221$.
However, computer simulations suggest that (for any $h$) we always have $\lfloor \alpha n - c \rfloor = \lfloor r_n \rfloor$.

\section{Acknowledgement}
\noindent
The authors gratefully acknowledge the organization of the Lorentz Center,
Leiden, workshop {\it Aperiodic Patterns in Crystals, Numbers and Symbols},
in particular Robbert Fokkink, for bringing them together in a stimulating
atmosphere.

The third author was supported by the ANR project `Dyna3S' (ANR-13-BS02-0003).

\end{document}